\documentclass{amsart}

\usepackage{color}
\usepackage{amsmath,amssymb,graphicx,hyperref}
\usepackage[shortalphabetic]{amsrefs}

\renewcommand\MR[1]{%
    \relax\ifhmode\unskip\spacefactor3000 \space\fi
    MR\nolinebreak[3]\hspace{.16667em}#1%
}

\renewcommand{\baselinestretch}{1.3}

\theoremstyle{plain}
\newtheorem{theorem}{Theorem}[section]
\newtheorem{lemma}[theorem]{Lemma}
\newtheorem{proposition}[theorem]{Proposition}
\newtheorem{corollary}[theorem]{Corollary}

\theoremstyle{remark}

\newtheorem{conjecture}{Conjecture}

\theoremstyle{definition}

\newcommand{\dis}{\displaystyle}

\newcommand{\Z}{\mathbb{Z}}

\newcommand{\al}{\alpha}
\newcommand{\be}{\beta}

\newcommand{\la}{\lambda}
\newcommand{\D}{\Delta}
\newcommand{\Si}{\Sigma}
\newcommand{\So}{\Sigma_0}

\newcommand{\I}{^{-1}}
\newcommand{\gen}{\langle}
\newcommand{\by}{\rangle}

\DeclareMathOperator{\SL}{SL}

\DeclareMathOperator{\GL}{GL}

\DeclareMathOperator{\Aut}{Aut}
\DeclareMathOperator{\lk}{lk}
\DeclareMathOperator{\rk}{rk}

\numberwithin{equation}{section}

\begin{document}

\title[Anisotropic groups admitting no split spherical BN-pairs]{Some reductive anisotropic groups that admit no non-trivial split spherical BN-pairs}

\author{Peter Abramenko}
\address{Department of Mathematics\\
University of Virginia\\
Charlottesville, VA 22904}
\email{pa8e@virginia.edu}

\author{Matthew C. B. Zaremsky}
\address{Department of Mathematics\\
University of Virginia\\
Charlottesville, VA 22904}
\email{mcz5r@virginia.edu}

\begin{abstract}
We prove, for any infinite field $k$, that any virtually trivial split spherical $BN$-pair in the group $G(k)$ of $k$-rational points of a reductive $k$-group $G$ is already trivial. We then inspect the case when $G$ is $k$-anisotropic and show that in many situations $G(k)$ admits no non-trivial split spherical $BN$-pairs. This improves results and contributes to a conjecture of Caprace and Marquis, which can be viewed as a converse to a well-known result of Borel and Tits.
\end{abstract}

\maketitle

\section{Introduction}\label{sec:intro}

We prove a variety of results inspired by a conjecture of Caprace and Marquis, formulated in \cite{caprace09}.

\begin{conjecture}\label{conj1}
Let $G$ be a reductive algebraic $k$-group that is anisotropic over $k$. Then every split spherical $BN$-pair of $G(k)$ is trivial \cite{caprace09}.
\end{conjecture}

Here and in the following, a reductive group is always connected by definition.
A $BN$-pair $(B,N)$ of a group $G$ is called \emph{saturated} if $T:=B\cap N$ is equal to $\dis\bigcap_{w\in W}wBw\I$, and \emph{split} if it is saturated and $B$ decomposes as $B=U\rtimes T$, where $U$ is nilpotent. We say $(B,N)$ is \emph{trivial} if $B=G$, or equivalently if the building $\D=\D(G,B)$ is trivial. See \cite{abr08} for the relevant background on buildings. Conjecture~\ref{conj1} can be phrased as a converse to a well-known result of Borel and Tits \cite{bortits65}. Namely, for any reductive algebraic $k$-group $G$ that is \emph{isotropic} over $k$,
the group of $k$-rational points $G(k)$ admits a canonical (non-trivial) split spherical $BN$-pair. In their paper Caprace and Marquis focus on the following weaker conjecture.

\begin{conjecture}\label{conj2}
Let $G$ be a reductive algebraic $k$-group that is anisotropic over $k$. Then every split spherical $BN$-pair of $G(k)$ is virtually trivial \cite{caprace11}.
\end{conjecture}

We say a $BN$-pair $(B,N)$ of $G$ is \emph{virtually trivial} if $[G:B]<\infty$, or equivalently if the building $\D=\D(G,B)$ is finite. In \cite{caprace11} Conjecture~\ref{conj2} is shown to hold if $k$ is either local or perfect. In the present context we prove that for any infinite field $k$, if $G(k)$ consists of semisimple elements then Conjecture~\ref{conj1} holds. We also discover that, contrary to expectations, Conjecture~\ref{conj1} actually follows from Conjecture~\ref{conj2}, and so the two conjectures are equivalent. As explained in \cite{caprace11}, Conjecture~\ref{conj1} is related to a conjecture of Prasad and Rapinchuk that has so far only been proved in the $D^{\times}$ case,
namely that any finite quotient of an anisotropic reductive group must be solvable \cite{rapinchuk02}.

We phrase our main results without explicit reference to anisotropic groups, for the sake of full generality. By \cite{borel91}*{Corollary~18.3}, for any infinite field $k$, $G(k)$ is Zariski-dense in $G$. Also, if $G$ is $k$-anisotropic then in many cases $G(k)$ contains no non-trivial unipotent elements, or even consists solely of semisimple elements. These are the only properties we will need to prove our main results, though one should keep in mind the example of $H=G(k)$, and we will explicitly state our results with respect to the conjectures at the end.

The main results here are the following:

\begin{theorem}\label{refine_capr_marq}
Let $G$ be a reductive group over an algebraically closed field. Let $H$ be a subgroup of $G$ that is Zariski-dense in $G$. Then any virtually trivial split $BN$-pair of $H$ is trivial.
\end{theorem}

In particular if we take $H=G(k)$ for some infinite field $k$, Conjecture~\ref{conj2} implies Conjecture~\ref{conj1}. Note however that Theorem~\ref{refine_capr_marq} still applies even if $G(k)$ is isotropic.

\begin{theorem}\label{no_unip_thm}
Let $G$ be a reductive group over an algebraically closed field. Let $H$ be a subgroup that is Zariski-dense in $G$ and contains no non-trivial unipotent elements. Then any split spherical $BN$-pair of $H$ having irreducible rank $\geq2$ is trivial.
\end{theorem}

\begin{theorem}\label{all_ss_thm}
Let $G$ be a reductive group over an algebraically closed field. Let $H$ be a subgroup that is Zariski-dense in $G$ and consists only of semisimple elements. Then any split spherical $BN$-pair of $H$ is trivial.
\end{theorem}

In Section~\ref{sec:v_triv_to_triv} we prove Theorem~\ref{refine_capr_marq}. In particular if $k$ is local or perfect then by \cite{caprace11} Conjecture~\ref{conj2} holds, and so we conclude that even Conjecture~\ref{conj1} holds. We next present a proof of Theorem~\ref{no_unip_thm} in Section~\ref{sec:big_rank}, and prove Theorem~\ref{all_ss_thm} in Section~\ref{sec:any_rank}. Note that applying Theorem~\ref{all_ss_thm} with $H=G(k)$, the conclusion of Theorem~\ref{no_unip_thm} can be sharpened in some cases, e.g. if $k$ is perfect. In this way Conjecture~\ref{conj1} is shown to hold for the case when $k$ is perfect using a different line of attack than in \cite{caprace11}. We also establish Conjecture~\ref{conj1} for certain $G$ with no restriction on $k$, for example if $G(k)$ is the multiplicative group of a division algebra; see Section~\ref{sec:div_algs_and_conclusions}.

We mention that a much more general program is carried out by Prasad in \cite{prasad11}, which in particular establishes all the results proved here, and more. Our methods are very different, however, and could possibly be useful in other contexts, so the present work is still of interest.

\section{Virtually trivial implies trivial}\label{sec:v_triv_to_triv}

In this section we prove Theorem~\ref{refine_capr_marq}. The following key lemma plays an important role here as well as in Proposition~\ref{T_v_triv_implies_triv} below.
The first part of the proof below mimics parts of the proof of Theorem~4 in \cite{caprace11}.

\begin{lemma}\label{key_lemma}
Let $G$ be a reductive group over an algebraically closed field $K$, and let $H$ be any subgroup of $G$. Suppose that $H$ possesses a split $BN$-pair $(B=U\rtimes T,N)$ such that $G=\overline{B}Z(G)$, where $\overline{B}$ is the Zariski closure of $B$ in $G$. Then $U$ is trivial.
\end{lemma}
\begin{proof}
We first show that $U$ is finite. Since $U$ is nilpotent and normal in $B$, we know that $\overline{U}$ is nilpotent and normal in $\overline{B}$ \cite{borel91}*{Section~2.1}. By hypothesis $G=\overline{B}Z(G)$, so in fact $\overline{U}$ is normal in $G$. This implies that the identity component $\overline{U}^0$ is contained in the radical of $G$, which coincides with $Z(G)^0$ \cite{borel91}*{Proposition~11.21}. Now as in the proof of \cite{caprace11}*{Theorem~4} we have
$$[U:U\cap Z(G)]\leq[U:U\cap\overline{U}^0]=[U\overline{U}^0:\overline{U}^0]\leq[\overline{U}:\overline{U}^0]<\infty.$$
Also, if $u\in U\cap Z(G)$, then for any $g\in H$ we have $ugC=guC=gC$ where $C$ is the fundamental chamber in $\D=\D(H,B)$. Of course every chamber of $\D$ is of the form $gC$ for some $g\in H$, so $u$ acts trivially on $\D$. But $T$ contains the kernel of the action and $U\cap T=\{1\}$, so in fact $u=1$. We conclude that $U$ is finite.

We now want to show that $U$, or equivalently $\D$, is even trivial. Since $U$ is finite, $U=\overline{U}$, which we know is normal in $G$. Thus $U$ is normal in $H$. Now suppose $S\neq\emptyset$. Let $s\in S$, so by the $BN$-axioms $sBs\not\leq B$. But $s$ normalizes $T$, and since $U$ is normal in $H$ we know that $s$ also normalizes $U$, so this is impossible. We conclude that $S=\emptyset$, so $N=T\leq B$ and in fact $B=H$. Since the chambers of $\D$ are in one-to-one correspondence with $H/B$, we conclude that $\D$, and thus $U$, is trivial.
\end{proof}

\begin{lemma}\label{v_triv_implies_B_dense}
Let $G$ be a connected linear algebraic group, $H$ a Zariski-dense subgroup of $G$ and $B$ any subgroup of $H$. Then either $B$ has infinite index in $H$ or $B$ is Zariski dense in $G$.
\end{lemma}
\begin{proof}
Assume that $B$ has finite index in $H$. Let $h_1,\dots,h_n$ be a set of coset representatives for $H/B$, and let $\overline{B}$ be the Zariski closure of $B$ in $G$. Then $G=\overline{H}$ is the union of the cosets $h_1\overline{B},\dots,h_n\overline{B}$. Since $G$ is connected, $G=h_i\overline{B}$ for some $i$, and hence also $G=\overline{B}$.
\end{proof}

We are now in a position to prove Theorem~\ref{refine_capr_marq}.

\begin{proof}[Proof of Theorem~\ref{refine_capr_marq}]
Let $G$ be a reductive group and $H$ a Zariski-dense subgroup. Let $(B=U\rtimes T,N)$ be a split $BN$-pair of $H$ such that $[H:B]<\infty$. By Lemma~\ref{v_triv_implies_B_dense} $B$ is Zariski dense in $G$, and by Lemma~\ref{key_lemma}, $U$ is trivial.
\end{proof}

In particular if $H=G(k)$ this shows that Conjectures~\ref{conj1} and \ref{conj2} are equivalent. We immediately get the following:

\begin{corollary}[Extension of Theorems~3 and 4 from \cite{caprace11}]\label{extend_cap_marq}
Let $k$ be either a local field or a perfect field. Let $G$ be a reductive $k$-group that is anisotropic over $k$. Then $G(k)$ does not admit any non-trivial split spherical $BN$-pairs.\qed
\end{corollary}

\section{Unipotent-free groups and split spherical BN-pairs without rank 1 factors}\label{sec:big_rank}

Let $G$ be a group with a $BN$-pair $(B,N)$ of spherical type $(W,S)$. In this section, our standing assumption is that the Coxeter system $(W,S)$ has no direct rank 1 factors,
which means that the corresponding Coxeter diagram has no isolated nodes. As usual, we set $T = \dis\bigcap_{w\in W}wBw\I$.
By the main result of \cite{demedts05}, the existence of a splitting $B=U\rtimes T$ with nilpotent $U$ implies
that the building $\D=\D(G,B)$ is Moufang and $U=U_+$, the group generated by all the root groups $U_{\al}$ for $\al\in\Phi_+$. Here $\Phi$ is the root system of a fixed fundamental apartment corresponding to $N$, and $\Phi_+$ is the set of all $\al \in \Phi$ that contain the fundamental chamber corresponding to $B$.
In this section we prove Theorem~\ref{no_unip_thm} by inspecting the root structure.

Recall that the concepts of spherical Moufang buildings and spherical RGD systems are equivalent \cite{abr08}*{Example 7.83 and Theorem 7.116}.
The following is a general lemma about spherical (or just 2-spherical) RGD systems.

\begin{lemma}\label{noncomm_roots}
Let $(G,(U_{\al})_{\al\in\Phi}),T)$ be an RGD system of spherical type $(W,S)$, such that the Coxeter diagram of $(W,S)$ has no isolated nodes. Let $\Phi_+$ be a choice of positive roots and $\Pi\subseteq\Phi_+$ a choice of simple roots. Then for any simple root $\al\in\Pi$, there exists a simple root $\be$ such that no non-trivial elements of $U_{\be}$ commute with any non-trivial elements of $U_{\al}$.
\end{lemma}
\begin{proof}
Let $U_+$ and $B_+=U_+T$ be the usual subgroups. We consider the Moufang building $\D=\D(G,B_+)$ with fundamental chamber $C$ and fundamental apartment $\So$.
Let $\al=\al_s$ be the simple root corresponding to $s\in S$, so $C \in\al$ but $sC \not\in\al$. Hence there exists a panel $P$ of $C$ in the boundary of $\al$, and $P$ has cotype $s$. Choose $t\in S\backslash\{s\}$ such that $s$ and $t$ are connected in the Coxeter diagram, and let $Q$ be the panel of $C$ of cotype $t$. Denote by $\be$ the root containing $C$ and having $Q$ in its boundary, i.e. $\be$ is the simple root $\al_t$. We set $A:=P\cap Q$, so $A$ has cotype $\{s,t\}$.

Let $\D':=\lk_{\D}(A)$ be the link of $A$ in $\D$, and set $\al':=\al\cap\D'$, $\be':=\be\cap\D'$, $C':=C\cap\D'$, and $\So':=\So\cap\D'$. By \cite{abr08}*{Proposition~7.32}, $\D'$ is strictly Moufang and we can identify $U_{\al'}$ and $U_{\be'}$ with $U_{\al}$ and $U_{\be}$ respectively, via the natural restriction map. In fact, since $A$ has codimension 2, $\D'$ is a generalized Moufang $n$-gon, and since $s$ and $t$ are connected in the Coxeter diagram we know that $n>2$.

Let $1\neq a\in U_{\al'}$ and $b\in U_{\be'}$ and suppose $ab=ba$. Then $ba\be'=ab\be'=a\be'$ and $C' \in a\be'$ is fixed by $b$, so $b$ fixes $a\be'$ pointwise.
Since $\al'\cap\be'=\{C'\}$, $n>2$, and $a\neq 1$, we know that $\be'\cup a\be'$ contains a pair of opposite chambers. If $\Si'$ is the apartment containing these, clearly $b$ fixes $\Si'$ pointwise. But $\Si'$ also contains some chamber $D'$ of $\be'$ that has no panels in $\partial\be'$. Since $b$ fixes $\Si'$ and all chambers of $\D'$ adjacent to $D'$, by the rigidity theorem \cite{abr08}*{Corollary~5.206} we conclude that $b$ is the identity on $\D'$, and hence $b=1$. See Figure~\ref{fig} for an idea of the $n=3$ situation.

Now we return to the original building $\D$. If $1\neq a\in U_{\al}$ and $b\in U_{\be}$ commute, then by the above argument $b$ acts trivially on $\lk_{\D}(A)$. In particular $b$ fixes any simplex joinable to $Q$. Since $Q\in\partial\be$ this implies that $b=1$, and the result follows.
\end{proof}

\begin{figure}[h]
\caption{Example for $n$=3}
\label{fig}
\includegraphics[scale=.4]{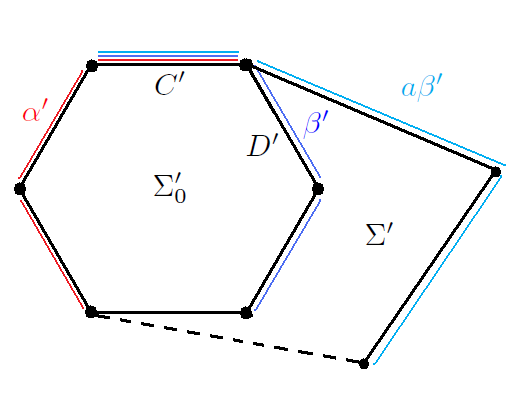}
\end{figure}

\begin{lemma}\label{nilp_implies_vabln}
Let $H$ be any nilpotent linear group having no non-trivial unipotent elements. Then $H$ is virtually abelian.
\end{lemma}
\begin{proof}
Assume $H\leq\GL_n(K)$ for some algebraically closed field $K$. Since $H$ is nilpotent, so is the Zariski closure $\overline{H}$. By \cite{borel91}*{Theorem~10.6(3)} then, the connected component $\overline{H}^0$ decomposes as a direct product of its semisimple and unipotent parts, $\overline{H}^0=(\overline{H}^0)_s\times(\overline{H}^0)_u$. Since $(\overline{H}^0)_s$ is abelian and the product is direct, we know that $[\overline{H}^0,\overline{H}^0]=[(\overline{H}^0)_u,(\overline{H}^0)_u]$. Now set $H':=H\cap\overline{H}^0$. Of course $[H',H']\leq H$, and also we see that $[H',H']\leq[(\overline{H}^0)_u,(\overline{H}^0)_u]$, which consists of unipotent elements. Since $H$ has no non-trivial unipotent elements, in fact $H'$ is abelian. Also, $[H:H']<\infty$ so indeed $H$ is virtually abelian.
\end{proof}

\begin{proof}[Proof of Theorem~\ref{no_unip_thm}]
Let $G$ be a reductive group and $H$ a Zariski-dense subgroup such that $H\cap G_u=\{1\}$. Let $(B=U\rtimes T,N)$ be a split spherical $BN$-pair of $H$ of type $(W,S)$ without rank 1 factors. We claim that $B=H$.
Since the Coxeter diagram of $(W,S)$ has no isolated nodes, we can use Lemma~\ref{noncomm_roots}. Let $\D=\D(H,B)$ with fundamental apartment $\So$ and fundamental chamber $C$. Denote the root system by $\Phi$ and the root groups by $U_{\al}$, so $U=U_+=\gen U_{\al}\mid\al\in\Phi_+\by$. We know that $U$ is nilpotent, and so by Lemma~\ref{nilp_implies_vabln} we can choose $V\leq U$ abelian such that $[U:V]<\infty$.

Given any system of positive roots $\Phi_+$ and any simple root $\al \in \Phi_+$, Lemma~\ref{noncomm_roots} shows that we can choose a simple root $\be$ such that no non-trivial elements of $U_{\al}$ commute with any non-trivial elements of $U_{\be}$. Since $V$ is abelian, this shows that either $U_{\al}\cap V=\{1\}$ or $U_{\be}\cap V=\{1\}$. In either case at least one of the root groups is finite. However, by the classification of Moufang polygons \cite{tits_weiss02}, if one of the root groups is finite then they must \emph{all} be finite. Since $\D$ is Moufang, this implies that $\D$ is \emph{locally finite}, i.e. each panel in $\D$ is contained in finitely many chambers. Thus for any $w\in W$ there are finitely many chambers of $\D$ at Weyl distance $w$ from $C$. Since $W$ itself is finite, we conclude that $\D$ is finite. By Theorem~\ref{refine_capr_marq} it is even trivial, and so the $BN$-pair is trivial.
\end{proof}

We now take a moment to discuss some examples of unipotent-free anisotropic groups. Let $D$ be a finite dimensional central $k$-division algebra and suppose $x\in D$ is unipotent. Then $x-1$ is a nilpotent element of $D$, which since $D$ is a division algebra implies that $x-1=0$ and $x=1$. In particular the anisotropic groups $D^{\times}$ and $\SL_1(D)$ are unipotent-free. The theorem tells us they admit no split spherical $BN$-pairs of rank greater than 1, though as we will see in the last section we can eliminate the rank-1 case as well.

Also, if $k$ is perfect then any anisotropic $G(k)$ is unipotent-free, as explained in \cite{caprace11}*{Proposition~4.1}. Thus the proof of Theorem~\ref{no_unip_thm} constitutes an alternate proof of Conjecture~\ref{conj1} in case $k$ is perfect, at least for $BN$-pairs of rank $\geq2$. Again, in the next section the rank-1 case will also be eliminated.

Note that the only time we used our precise setup in the proof of Theorem~\ref{no_unip_thm} was to see that if $U$ is nilpotent then it is already virtually abelian. In fact, the same proof yields the following

\begin{proposition}\label{v_abln_implies_triv}
Let $G$ be a group with a split spherical $BN$-pair $(B=U\rtimes T,N)$ of type $(W,S)$ without rank 1 factors. If $U$ is virtually abelian then $U$ is finite.\qed
\end{proposition}

It seems likely to us that this last proposition can be applied to other reductive anisotropic groups as well, even perhaps ones with unipotent elements, in which case we would conclude, applying Theorem~\ref{refine_capr_marq} again, that these $BN$-pairs are even trivial.

\section{Groups consisting of semisimple elements and split BN-pairs of any rank}\label{sec:any_rank}

In this section, we prove Theorem~\ref{all_ss_thm}. It is based on a more general criterion (see Proposition~\ref{T_v_triv_implies_triv}) that might even be applicable in more generality.

\begin{proposition}\label{B_not_v_nilp}
Let $G$ be a nonabelian reductive group over an algebraically closed field $K$, $H$ a Zariski-dense subgroup of $G$ and $A\leq H$ such that $|A\backslash H/A|<\infty$. Then $A$ is not virtually nilpotent.
\end{proposition}

We will not individually cite every result quoted in the proof, but all the details can be found in Chapters~13 and 14 of \cite{borel91}, unless otherwise cited.

\begin{proof}

If $|A\backslash H/A|<\infty$ and $\overline{H} = G$ is connected, then there exists a double coset $AhA$ ($h \in H$) that is Zariski-dense in $G$.
Now if $A$ were virtually nilpotent, also its Zariski closure $\overline{A}$ would be, and there would exist a closed connected nilpotent subgroup $M$ of $\overline{A}$ of finite index.
So $\overline{A}h\overline{A}$ would be a finite union of double cosets modulo $M$, and one of these would have to be dense in $G$. Hence our claim will follow if we can show that $MgM$ is not
dense in $G$, for any closed connected nilpotent subgroup $M$ of $G$ and any $g \in G$. For this it suffices to show $\dim{MgM} < \dim{G}$ (where we set $\dim{MgM} = \dim{\overline{MgM}}$),
and this is what we are going to do now.

First consider the case when $G$ is \emph{semisimple}. Choose a Borel subgroup $B$ of $G$ that contains $M$, and denote by $U$ the unipotent radical of $B$. As a closed connected nilpotent group,
$M$ is the direct product of a torus $T'$ and a unipotent subgroup $U'$, see \cite{borel91}*{Theorem~10.6}. Necessarily, $U'$ is a subgroup of $U$. Choose a maximal torus $T$ of $B$ that
contains $T'$. Let $\Phi$ be the root system associated to $G$ and $T$. Then $B$ determines a positive system $\Phi_+$ in $\Phi$. Recall that
$\dim{U} = |\Phi_+|$ and $\dim{G} = \dim{T} + |\Phi|$.

We want to show that $\dim{M} \leq \dim{U}$, which is clear if $T' = \{1\}$. The simple idea of the following argument is to deduce $\dim{U'} \leq \dim{U} - \dim{T'}$ from the fact that the
product $T' \times U'$ is direct. Order the positive roots $\Phi_+=\{\al_1,\dots,\al_m\}$, and for each $1\leq i\leq m$ denote the corresponding root group by $U_i:=U_{\al_i}$. Then the map $U_1\times\cdots\times U_m\rightarrow U$ given by multiplication is a bijection \cite{borel91}*{Proposition~14.4}. Next choose isomorphisms $x_i:(K,+)\rightarrow U_i$ for all $i$.
It follows (see \cite{borel91}*{Section~10.10}) that $tx_i(\la)t\I=x_i(\al_i(t)\la)$ for $t\in T$, $\la\in K$. Then for any $u \in U'$, there exist uniquely determined
$\la_i\in K$ such that $\dis u=\prod_{i=1}^m x_i(\la_i)$, and $tut\I = \prod_{i=1}^m x_i(\al_i(t)\la_i)$ for all $t \in T$. Since $tut\I=u$ for all $t \in T'$, we obtain

$$\prod_{i=1}^m x_i(\al_i(t)\la_i)=\prod_{i=1}^m x_i(\la_i) \quad \mbox{for }  t \in T'.$$
By uniqueness, for each $i$ such that $\la_i\neq0$ we must have $T'\leq\ker\al_i$. In particular we get that
\begin{align}\label{bound_on_U'}
U'\leq\prod_{i~:~T'\leq\ker\al_i}U_i.
\end{align}
Now consider the character groups $X(T)$ and $X(T')$. Let $\pi:X(T)\twoheadrightarrow X(T')$ be the restriction map, so $\ker\pi=\{\chi\in X(T)\mid T'\leq\ker\chi\}$. Also,
$$\dim{T}=\rk(X(T))=\rk(\ker\pi)+\rk(X(T'))=\rk(\ker\pi)+\dim{T'}$$
so $\rk(\ker\pi)=\dim{T}-\dim{T'}$.

Now let $\ell:=\dim{T}=\rk(X(T))$. By \cite{borel91}*{Theorem~13.18(3)} $\gen\Phi\by_{\Z}$ has rank $\ell$, so without loss of generality the first $\ell$ positive roots $\al_1,\dots,\al_{\ell}$
can be chosen to be linearly independent.
Since $\rk(\ker\pi)=\dim{T}-\dim{T'}$, only $\ell-\dim{T'}$ of the roots $\al_1,\dots,\al_{\ell}$ can be contained in $\ker\pi$. In particular at least $\dim{T'}$ of them are \emph{not} contained in $\ker\pi$. By \ref{bound_on_U'} we conclude that the dimension of $U'$ is less than or equal to $\dim{U}-\dim{T'}$, so indeed $\dim M = \dim{T'} + \dim{U'} \leq\dim U$.
Since $T$ is non-trivial (if $B = U$, then also $G = B = U$ by \cite{borel91}*{Corollary~11.5}, which contradicts our assumptions), it follows as claimed that
$$\dim{MgM} \leq 2\dim{M} \leq 2\dim{U} < \dim{T} + 2\dim{U} = \dim{G}.$$

We now consider the case when $G$ is \emph{reductive}. By \cite{borel91}*{Proposition~14.2}, $G$ decomposes as $G=SG_1$ with $S:=Z(G)^0$ and semisimple $G_1:=[G,G]$. Note that $G_1$ has the same root system $\Phi$ as $G$. Since $G$ is not abelian,  $G_1$ is non-trivial and hence not nilpotent, since semisimple. Choose a maximal torus $T$ of $G$; it has to contain $S$ since $ST$ is a torus. If we had $S = T$, then this were the unique maximal torus of $G$, implying (again by \cite{borel91}*{Corollary~11.5}) that $G$ is nilpotent, a contradiction.
Now let $M$ be a closed connected nilpotent subgroup of $G$, which without loss of generality contains $S$.
Obviously $M=S(M\cap G_1)$, and since $G_1$ is semisimple, the previous case implies $\dim(M\cap G_1)\leq|\Phi_+|$. Thus
$$\dim(MgM)=\dim((M\cap G_1)gM)\leq \dim(M\cap G_1)+\dim{M}\leq 2|\Phi_+|+\dim{S}.$$
Also, $G$ has dimension $2|\Phi_+|+\dim{T}>2|\Phi_+|+\dim{S}$, so we conclude that $MgM$ is not dense in $G$.
\end{proof}

\begin{proposition}\label{T_v_triv_implies_triv}
Let $G$ and $H$ be as in Proposition~\ref{B_not_v_nilp}. Let $(B=U\rtimes T,N)$ be a split spherical $BN$-pair of $H$, and suppose that the kernel of the action of $H$ on the building $\Delta=\Delta(H,B)$ has finite index in $T$. Then $U$ and $\Delta$ are trivial.
\end{proposition}
\begin{proof}
Denote by $Q$ the kernel of the action of $H$ on $\D$, so $[T:Q]<\infty$. Since $Q$ is normal in $H$, if we pass to the Zariski closure we get that $\overline{Q}$ is normal in $\overline{H}=G$.
Consider the canonical projection $\pi : G \twoheadrightarrow G/\overline{Q}$. It follows from \cite{borel91}*{Corollary~14.11} that $G/\overline{Q} = \pi(G)$ is reductive. Also $\pi(H)$ is
dense in $G/\overline{Q}$ since $G/\overline{Q} = \pi(\overline{H}) \subseteq \overline{\pi(H)}$. Since $(B,N)$ is a spherical $BN$-pair in $H$, $|B\backslash H/B|<\infty$, which implies
$|\pi(B)\backslash \pi(H)/\pi(B)|<\infty$. Next we note that $B/Q$ is virtually nilpotent since it is a semidirect product of $U$ and $T/Q$, and $|T/Q| < \infty$ by assumption. But then also
$\pi(B)$, which is isomorphic to $B/(B \cap \overline{Q})$ and hence to a quotient of $B/Q$, is virtually nilpotent. Combining all these facts, Proposition~\ref{B_not_v_nilp} now implies that
$G/\overline{Q}$ is \emph{abelian}. As above we have a decomposition $G=[G,G]Z(G)$ where $[G,G]$ is the commutator subgroup. Since $G/\overline{Q}$ is abelian, clearly $[G,G]\leq\overline{Q}$, so $G=\overline{Q}Z(G)$. Since $\overline{Q}\leq\overline{B}$, we can now apply Lemma~\ref{key_lemma} to conclude that $U$ and $\D$ are trivial.
\end{proof}

\begin{lemma}\label{rigidity_lemma}
Let $G$ be any linear algebraic group. Let $U$ be a nilpotent subgroup of $G$ consisting of semisimple elements, and $T$ any subgroup of $N_G(U)$. Then $[T:C_T(U)]<\infty$.
\end{lemma}
\begin{proof}
We first claim that $[U:Z(U)]<\infty$. Of course by Lemma~\ref{nilp_implies_vabln} we already know that $U$ is virtually abelian, but in the present context we can do even better. Indeed, the Zariski closure $\overline{U}$ is nilpotent, and so by \cite{borel91}*{Proposition~12.5} the semisimple part of the connected component $(\overline{U}^0)_s$ is central in $\overline{U}$. In particular $U\cap(\overline{U}^0)_s$ is central in $U$. But since $U$ consists of semisimple elements, $U\cap(\overline{U}^0)_s=U\cap\overline{U}^0$. This has finite index in $U$ and so indeed $[U:Z(U)]<\infty$.

Let $\{u_1,\dots,u_n\}$ be a set of coset representatives of $U/Z(U)$. For any $1\leq i\leq n$, the group $V_i:=\gen u_i,Z(U)\by$ is abelian and consists of semisimple elements, i.e. is diagonalizable. In particular by rigidity \cite{borel91}*{Corollary~8.10(2)}, we see that $[T:C_T(V_i)]<\infty$. This implies that the group $\dis\bigcap_{i=1}^nC_T(V_i)$ has finite index in $T$. But if $t\in T$ centralizes every $V_i$ then of course $t$ centralizes $U$, so we conclude that indeed $[T:C_T(U)]<\infty$.
\end{proof}

We are now in a position to prove Theorem~\ref{all_ss_thm}. From now on $G$ is a reductive group over an algebraically closed field, and $H$ be a subgroup that is Zariski-dense in $G$ and consists only of semisimple elements.

\begin{proof}[Proof of Theorem~\ref{all_ss_thm}]
Let $(B=U\rtimes T,N)$ be a split spherical $BN$-pair of $H$, and let $\D=\D(H,B)$. We want to show that $\D$ is trivial. Let $\So$ be the fundamental apartment in $\D$, with fundamental chamber $C$. Let $D$ be the chamber opposite $C$ in $\So$, so any other chamber opposite $C$ in $\D$ is of the form $uD$ for some $u\in U$ (since $H$ acts strongly transitively on $\D$ and $B = UT$).
Now, if $t\in T$ commutes with every element of $U$ then clearly $t(uD)=utD=uD$. Thus the centralizer $C_T(U)$ of $U$ in $T$ is contained in the kernel of the action. Since $U$ is nilpotent and, by virtue of being contained in $H$, consists of semisimple elements, by Lemma~\ref{rigidity_lemma} in fact $[T:C_T(U)]<\infty$. In particular $T$ acts virtually trivially on $\D$, and so by Proposition~\ref{T_v_triv_implies_triv} $\D$ is trivial.
\end{proof}

\begin{corollary}\label{ss_cor}
Let $k$ be an infinite field. Let $G$ be reductive $k$-group such $G(k)$ consists of semisimple elements. Then $G(k)$ admits no non-trivial split spherical $BN$-pairs.\qed
\end{corollary}

In particular Conjecture~\ref{conj1} holds in case $G(k)$ consists of semisimple elements. If $G$ is $k$-anisotropic and $k$ is perfect then this is the case, by \cite{borel91}*{Proposition~4.2(5)} and the fact discussed earlier that $G(k)$ has no non-trivial unipotent elements. We thus have a second, different proof that Conjecture~\ref{conj1} holds when $k$ is perfect.

\section{Division algebras, and some conclusions}\label{sec:div_algs_and_conclusions}

As promised, in case $H=G(k)$ is the multiplicative or norm-1 group of a finite dimensional $k$-division algebra, we can also eliminate the rank-1 case. By the proof of Theorem~\ref{all_ss_thm} it suffices to show that if $(B=U\rtimes T,N)$ is any split spherical $BN$-pair of $H$, then $[T:C_T(U)]<\infty$. We prove this in the following

\begin{lemma}\label{div_alg_lemma}
Let $H$ be the multiplicative or norm-1 group of a finite dimensional $k$-division algebra $D$. Let $U\leq H$ be nilpotent and $T\leq N_H(U)$. Then $[T:C_T(U)]<\infty$.
\end{lemma}
\begin{proof}
Clearly $k[U]$ is a division subalgebra of $D$. Then the canonical faithful representation $k[U]\hookrightarrow\GL(k[U])\cong\GL_r(k)$ for $r=\dim_kk[U]$ is irreducible. Hence also the restricted representation $U\hookrightarrow\GL_r(k)$ is faithful and irreducible. By \cite{sup58}*{Theorem~27} we conclude that $[U:Z(U)]<\infty$. Now consider the action of $T$ on $U$ by conjugation. Clearly $Z(U)$ is normalized by this action, so we get a homomorphism $T\rightarrow\Aut(U/Z(U))$. Let $T_0$ denote the kernel of this map, so $T_0$ has finite index in $T$. Choose a transversal $\{u_1,\dots,u_r\}$ of $U/Z(U)$. For any $i$, $T_0$ normalizes the abelian group $V_i=\gen u_i,Z(U)\by$. If $k[V_i]$ denotes the subfield of $D$ spanned by $V_i$, we get a homomorphism $\phi_i:T_0\rightarrow\Aut(k[V_i]|k)$,
the kernel of which has finite index in $T_0$ (since $k[V_i]|k$ is a finite extension). In particular
$$\bigcap_{i=1}^r\ker\phi_i$$
has finite index in $T_0$, and thus in $T$. Of course any element of this subgroup centralizes each $V_i$, and so centralizes $U$. We conclude that indeed $[T:C_T(U)]<\infty$.
\end{proof}

\begin{corollary}\label{div_alg_cor}
If $H$ is the multiplicative or norm-1 group of a division ring, then $H$ admits no non-trivial split spherical $BN$-pairs.\qed
\end{corollary}

To summarize, at this point Conjecture~\ref{conj1} stands completely proved in case $k$ is local or perfect, or if $G(k)$ consists of semisimple elements, or if $G(k)$ is the multiplicative or norm-1 group of a finite dimensional $k$-division algebra. Also the conjecture stands proved modulo rank-1 $BN$-pairs if $G(k)$ contains no non-trivial unipotent elements.

Lastly we note that we have two main tools now to demonstrate that a split spherical $BN$-pair $(B=U\rtimes T,N)$ of a reductive anisotropic group is trivial. Namely, if the rank is greater than 1 it suffices to show $U$ is virtually abelian, by Proposition~\ref{v_abln_implies_triv}, and if the rank is arbitrary it suffices to show that $[T:C_T(U)]<\infty$, by Proposition~\ref{T_v_triv_implies_triv}. It is our hope that these criteria could prove useful in eventually establishing the full conjecture.

\renewcommand{\baselinestretch}{1}

\end{document}